\documentclass{amsart}

\usepackage{a4wide,amsmath,amssymb}
\usepackage[toc,page]{appendix}

\usepackage[integrals]{wasysym}

\usepackage{url}

\usepackage{hyperref}

\usepackage{mathtools}

\usepackage{xcolor}

\newtheorem{thm}{Theorem}[section]
\newtheorem*{thm*}{Theorem}
\newtheorem{lemma}[thm]{Lemma}
\newtheorem{prop}[thm]{Proposition}
\newtheorem{cor}[thm]{Corollary}
\theoremstyle{remark}
\newtheorem{rem}[thm]{Remark}

\theoremstyle{definition}

\newtheoremstyle{Claim}{}{}{\itshape}{}{\itshape\bfseries}{:}{ }{#1}
\theoremstyle{Claim}

\newcommand{\E}{\mathbb{E}}
\newcommand{\R}{\mathbb{R}}

\newcommand{\dd}{\hspace{0.7pt}{\rm d}}

\makeatletter
\@namedef{subjclassname@2020}{\textup{2020} Mathematics Subject Classification}
\makeatother

\theoremstyle{plain}

\begin{document}

\title[]{Uniqueness of solutions to MFG systems \\ with large discount}

\author{Marco Cirant}
\address{Dipartimento di Matematica ``Tullio Levi-Civita'', Universit\`a degli Studi di Padova, 
via Trieste 63, 35121 Padova (Italy)}
\curraddr{}
\email{cirant@math.unipd.it}
\thanks{}
\author{Elisa Continelli}
\address{Dipartimento di Matematica ``Tullio Levi-Civita'', Universit\`a degli Studi di Padova, 
via Trieste 63, 35121 Padova (Italy)}
\curraddr{}
\email{elisa.continelli@unipd.it}

\date{\today}

\begin{abstract} We prove that solutions to a class of Mean Field Game systems with discount are unique provided that the discount factor is large enough, and the Lagrangian term is (proportionally) small enough. This identifies an asymptotic uniqueness regime that falls outside the usual ones involving monotonicity.
\end{abstract}

\maketitle

\section{Introduction}

	The main purpose of this paper is to discuss the uniqueness of solutions, for large (positive) $\lambda$, to the following MFG system of PDE
	\begin{equation}\label{MFGgen}\tag{${\rm MFG}_\lambda$}
		\begin{cases}
			-\partial_t u_\lambda-\Delta u_\lambda+\frac{1}{\lambda}H(\lambda Du_\lambda,x)+\lambda u=F(x,m_\lambda(t)),\quad &x\in \mathbb{R}^n,\,t\in (0,\infty),\\
			\partial_t m_\lambda-\Delta m_\lambda-\text{div}(m_\lambda D_pH(\lambda Du_\lambda,x))=0,\quad &x\in \mathbb{R}^n,\,t\in (0,\infty),\\
			m_\lambda(x,0)=m_0(x),\quad &x\in \mathbb{R}^n. 
		\end{cases}
	\end{equation}
	
	The analysis of the problem with large $\lambda$ is motivated by the work of Bardi and Cardaliaguet \cite{BC}, who showed that, as $\lambda \to \infty$, solutions of \eqref{MFGgen} converge to solutions of the nonlinear parabolic PDE
	\begin{equation}\label{MFGlim}\tag{${\rm MF}_\infty$}
		\begin{cases}
			\partial_t m-\Delta m-\text{div}(mD_pH(D_x F(x,m(t)),x))=0,\quad &x\in \mathbb{R}^n,\,t\in (0,\infty),\\
			m(x,0)=m_0(x),\quad &x\in \mathbb{R}^n. 
		\end{cases}
	\end{equation}
	
	This convergence result bridges two different approaches to the modelling of large population phenomena. On the one hand, Mean Field Games (MFG), introduced by Lasry and Lions \cite{LL07} and Huang, Caines and Malham\'e \cite{hcm}, that describe Nash equilibria among a continuum of indistinguishable rational agents. On the other hand, agent-based models, where individuals react to the population distribution according to a given rule; these models have a large literature, see for instance the survey \cite{CCPsurvey}. In MFG of the form \eqref{MFGgen}, $F$ plays the role of the running cost (of an individual with state $x$ against the population distribution $m(t)$) and $H$ is the Hamiltonian (related to the cost due to the control). Most importantly, the parameter $\lambda$ tunes at the same time two properties of the model: it is a time discount factor (corresponding to the $\lambda u_\lambda$ term) and it is inversely proportional to the cost of the controls. For the sake of clarity, players' trajectories are driven by the controlled SDE $\dd X_t = v_t \dd t + \sqrt 2 \dd B_t$, and they aim to minimize
	\[
	v \mapsto \E \int_0^\infty e^{-\lambda s}\left( \frac1\lambda L(v_s, X_s) + F(X_s, m_\lambda(s)) \right) \dd s,
	\]
	where the Lagrangian $L$ is conjugate to the Hamiltonian $H$. In other words, as $\lambda$ increases, agents become less interested in future events, while at the same time their ``reaction cost'' becomes cheaper. Eventually, in the $\lambda \to \infty$ limit, the model degenerates into another one where the individual rule of motion is given by $-D_pH(D_x F(x,m(t)),x)$. Here, rationality is absent, as potential future events have no effect on the population dynamics. We mention that other models involving large discount effects have been studied in the literature, see for instance \cite{BLL19,DHL17}. 
	
	If the MFG was potential, that is when $F = \delta_m \mathcal F$, this phenomenon would be analogous to the so-called Weighted Energy-Dissipation mechanism in Calculus of Variations, see for example \cite{RSSS, MS} and references therein. It is interesting to note that one may look at the problem \eqref{MFGgen} as a second-order (in time) approximation of the first-order (in time) problem \eqref{MFGlim}.
	
	Regarding \eqref{MFGlim}, under mild assumptions on $D_p H, D_x F$, one expects uniqueness of solutions: this is in fact a McKean-Vlasov equation, to which standard Cauchy-Lipschitz theory applies (see for example \cite{BKRS}). On the other hand, while \cite{BC} shows that solutions of \eqref{MFGgen} converge to solutions of \eqref{MFGlim} as $\lambda \to \infty$, it is not clear whether they are \textit{unique} when $\lambda$ is finite. The goal of the present paper is to to show that \eqref{MFGgen} has in fact unique solutions if $\lambda$ is large enough, ``borrowing uniqueness'' from the limit problem. This is stated as follows (for the more explicit version, see Theorem \ref{maingen}).
	
	\begin{thm*} Assume (F1), (F2) and (H1), (H2) below. Then, there exists $\Lambda > 0$ such that, if $\lambda \ge \Lambda$, then solutions to \eqref{MFGgen} are unique.
	\end{thm*}
	
	To be precise, we restrict our attention to solutions that satisfy \eqref{aprioribounds} below. These solutions enjoy nice growth properties and estimates (see Proposition \ref{estimatesapriori}). Existence in this class is obtained in Proposition \ref{basicest} by taking limits of MFG systems on the time horizon $[0,T]$, with final condition $u_\lambda(T) \equiv0$, in the limit $T \to \infty$. A result similar to the one in Proposition \ref{estimatesapriori} was already discussed in \cite{BC}; here we show the validity of additional uniform double-sided estimates on second-order derivatives. Uniqueness of solutions is then obtained within the set of functions satisfying \eqref{aprioribounds}. We cannot exclude the possibility that other solutions not satisfying these conditions could coexist.
	
	An interesting point of the uniqueness result is that the critical constant $\Lambda$ does not really depend on the viscosity. In particular, the diffusion parameter is just set to be equal to one for simplicity throughout the whole paper, but if the Laplacian terms were replaced by $\nu \Delta$, with $\nu > 0$, then $\Lambda$ would be independent of $\nu$.
	
	Uniqueness in MFG is known in special circumstances, that are generally related to monotone structures: these are mainly the so-called Lasry-Lions monotonicity \cite{LL07} and the displacement monotonicity \cite{MesMou}; see also \cite{MouZhang, BM24, GraMe}. Another uniqueness scenario is when some ``smallness'' assumption is in force, such as in \cite{BC18, BF, AM, CCDE}; such smallness is often related to the time-horizon, but this is somehow analogous to situations involving ``small interactions''. The uniqueness regime that is identified here is certainly very different from the one arising from monotonicity. It could be vaguely close to the one involving ``small'' parameters: in some sense, when $\lambda$ is large, there are terms in \eqref{MFGgen} that are expected to vanish, at least from an heuristic point of view. Despite that, the structure of the proof looks different from all the previous cases. 
	
	It is convenient to think of the term $-D_pH(\lambda Du_\lambda(t,x),x)$ in the second equation of \eqref{MFGgen}, driving the evolution of $m(t)$, as a nonlinear and \textit{non-local} in time function of $m_\lambda$; indeed, it is clear from the first equation that $Du_\lambda(t)$ depends on the whole ``future'' behaviour of $m_\lambda$ on the time interval $[t, \infty)$. As a preliminary result, we show (see in particular Proposition \ref{convunifgen}) that $Du_\lambda(t)$ is actually close to $DF(x,m_\lambda(t))$ (the notation $D F= D_x F$ and $D u= D_x u$ will be used from time to time): 
		$$\left| \lambda D u_\lambda(x,t)-D_xF(x,m_\lambda(t))\right| \lesssim \frac{1}{\sqrt{\lambda}},
		$$ 
	so $-D_pH(\lambda Du_\lambda(t,x),x)$ sees in fact $m_\lambda$ almost only through its values at time $t$. This allows to derive a a quantitative version of the convergence result in \cite{BC} and quantitative convergence of $m_\lambda$ to the solution of the McKean-Vlasov equation \eqref{MFGlim} (see Proposition \ref{quantitative}). To achieve then uniqueness of solutions, we need some further information. It turns out (see Proposition \ref{dudist}) that there exists a positive constant $\eta$ such that, given any two solutions $(u_\lambda^1, m_\lambda^1)$, $(u_\lambda^2, m_\lambda^2)$ to \eqref{MFGgen}, then
	\[
	 \| \lambda D u_\lambda^1(\cdot,t) - \lambda D u_\lambda^2(\cdot,t) \|_\infty  \lesssim
	  d_1(m_\lambda^1(t), m_\lambda^2(t)) + \\
	 \int_{t}^{\infty}  d_1(m_\lambda^1(t), m_\lambda^2(t)) e^{-(\lambda - \eta) (r-t)}  \dd r.
	\]
	This says rather precisely how $D u_\lambda(t)$, and hence $-D_pH(\lambda Du_\lambda(t))$, changes as $m_\lambda$ changes. It is clear that, for large values of $\lambda$, variations of $m_\lambda$ in the ``future'' are almost irrelevant. A delicate aspect in the derivation of this estimate is that the time horizon is infinite. The estimate is enough, combined with a sort of non-local in time Gr\"onwall type lemma (see Proposition \ref{groplus}), to derive the uniqueness result.
	
	\smallskip
	
	To conclude, we stress that non-uniqueness is a broad phenomenon in the theory of MFG, see for instance \cite{BF, CB18, C19, CT19, BZ20}. The uniqueness regime identified here could lead to further investigation. First, $F$ is assumed here to depend on $m(t)$ in a non-local sense ($F$ is in fact a regularizing functional of $m$, for instance of convolution type). We expect the argument to work along similar lines when $F$ is a local function of the density $m(x,t)$, in which case the presence of a nondegenerate diffusion would be way more important. Another question is related to the long-time behavior of \eqref{MFGgen}; in particular, we expect the couple $(u_\lambda(t), m_\lambda(t))$ to converge to a stationary profile as $t \to +\infty$. Finally, one could have situations where the limit agent-based model \eqref{MFGlim} has multiple solutions; we do not know if \eqref{MFGgen} could enjoy uniqueness in some of those cases (though it might be reasonable, as the latter is a sort of second-order regularization of the former). In this hypothetical framework, one could investigate selection phenomena for \eqref{MFGlim}.
		
\bigskip
\textbf{Acknowledgements.} The authors are member of the Gruppo Nazionale per l'Analisi Matematica, la Probabilit\`a e le loro Applicazioni (GNAMPA) of the Istituto Nazionale di Alta Matematica (INdAM). They are partially supported by the INdAM-GNAMPA projects 2025 and by the EuropeanUnion–NextGenerationEU under the National Recovery and Resilience Plan (NRRP), Mission 4 Component 2 Investment 1.1 - Call PRIN 2022 No. 104 of February 2, 2022 of Italian Ministry of University and Research; Project 2022W58BJ5 (subject area: PE - Physical Sciences and Engineering) “PDEs and optimal control methods in mean field games, population dynamics and multi-agent models”. M. C. has been partially funded by King Abdullah University of Science and Technology Research Funding (KRF) under award no. CRG2024-6430.6.

\section{Assumptions and preliminary results}

Unless otherwise specified, for functions $u$ defined on the space-time cylinder $\R^n \times [0,\infty)$ we will use the notation $\|u\|_\infty = \sup_{(x,t) \in \R^n \times [0,\infty)} |u(x,t)|$ and $\|u(\cdot, t)\|_\infty = \sup_{x \in \R^n} |u(x,t)|$.

\medskip

Suppose that $m_0\in \mathcal{P}_1(\mathbb{R}^n)$. $F:\mathbb{R}^n\times\mathcal{P}_1(\mathbb{R}^n)\rightarrow\mathbb{R}$ and $H:\mathbb{R}^n\times \mathbb{R}^n\rightarrow\mathbb{R}$; moreover, there exist constants $C_o\geq 1$ and $L > 0$ such that the following holds.
	\begin{itemize}
		\item[(F1)] $F(\cdot, m)$ is of class $C^{2,\alpha}$ for all $m \in \mathcal P_1(\mathbb{R}^n)$, and 
		\begin{equation*}
			\lvert F(x,m)\rvert \leq C_o(1+\lvert x\rvert),\quad \forall (x,m)\in \mathbb{R}^n\times\mathcal{P}_1(\mathbb{R}^n)
		\end{equation*}
		\item[(F2)] For all $ x,y\in \mathbb{R}^n$ and $m,\tilde{m}\in\mathcal{P}_1(\mathbb{R}^n)$,
		\begin{equation*}
			\lvert F(x,m)-F(y,\tilde{m})\rvert\leq L(\lvert x-y\rvert+d_1(m,\tilde{m})),
		\end{equation*}
		\begin{equation*}
			\lvert D_xF(x,m)-D_xF(y,\tilde{m})\rvert\leq L(\lvert x-y\rvert+d_1(m,\tilde{m})),
		\end{equation*}
		where $d_1(\cdot,\cdot)$ denotes the $1-$Wasserstein distance.
	\end{itemize}
	
	\begin{itemize}
		\item[(H1)] $H$ is of class $C^{2,\alpha}$, and it is convex with respect to the first variable.
		\item[(H2)] For all $p,x\in \mathbb{R}^n$, 
		\begin{equation*}
			-C_o\leq H(p,x)\leq C_o(1+\lvert p\rvert^2),\quad C_o^{-1} I_d \leq D^2_{pp}H(p,x)\leq C_o I_d ,
		\end{equation*}
		\begin{equation*}
			\lvert D_x H(p,x)\rvert+\lvert D^2_{px} H(p,x)\rvert\leq C_o(1+\lvert p\rvert) ,
		\end{equation*}
		\begin{equation*}
			\lvert D^2_{xx} H(p,x)\lvert\leq C_o.
		\end{equation*}
	\end{itemize}
	We start with some a priori estimates.
	\begin{prop}\label{estimatesapriori}
	Under the standing assumptions on $F$, $H$ and $m_0$, there is $\lambda_0\geq 1$, depending on $n$, $C_o$, $L$, such that, for any $\lambda \ge \lambda_0$, if $(u_\lambda,m_\lambda)$ is a classical solution to \eqref{MFGgen} satisfying 
	\begin{equation}\label{aprioribounds}
	\begin{gathered}
		\underset{(x,t) \in \R^n \times [0,\infty)}{\inf} u_\lambda(x,t)>-\infty, \quad \sup_{(x,t) \in \R^n \times [0,\infty)}\frac{u_\lambda(x,t)}{1+\lvert x\rvert}<\infty, \\ \lVert Du_\lambda (x,t)\rVert_\infty<\infty,\quad\lVert D^2u_\lambda (x,t)\rVert_\infty\leq1,
       \end{gathered}
	\end{equation}
	then
	\begin{align*}
		& \text{(i)} \quad  -\frac{C}{\lambda}\leq u_\lambda(x,t)\leq \frac{C}{\lambda}(1+\lvert x\rvert), \quad \forall x\in \mathbb{R}^n,\,\forall t\in [0,\infty), \\
		& \text{(ii)} \quad \lVert Du_\lambda\rVert_\infty\leq \frac{C}{\lambda}, \\
		& \text{(iii)} \quad \lVert D^2u_\lambda \lVert_\infty  \leq \frac{C}{\lambda},
	\end{align*}
	where $C$ is a universal constant depending only on $C_o$, $L$.
	\end{prop}
	It is remarkable that, in the above estimates (in particular the one regarding second derivatives), the diffusion plays almost no role; in particular, the constant $C$ is independent of the diffusion parameter.
	
	Note finally that, by (ii) in the previous proposition and (H2), we have that
	\begin{equation}\label{H3}
		\lvert D_x H(\lambda Du_\lambda(x,t),x)\rvert+\lvert D^2_{px} H(\lambda Du_\lambda(x,t),x)\rvert\leq C_o(1+C);
	\end{equation}
	this fact will be used several times throughout the paper.
	
	\begin{proof}[Proof of Proposition \ref{estimatesapriori}]
		Let us first prove (i). Fix $T>0$. Let us set $C_{u}:=\underset{(x,t) \in \R^n \times [0,\infty)}{\sup}\frac{u_\lambda(x,t)}{1+\lvert x\rvert}$. Note that, since $u_\lambda$ is bounded below, it holds $C_u\geq 0$. We define $w(x,t):=\frac{4C_o}{\lambda}(1+\lvert x\rvert^2)^{\frac{1}{2}}+\sqrt{2}C_{u}e^{-\frac{\lambda}{2}(T-t)}(1+\lvert x\rvert^2)^{\frac{1}{2}}$, $x\in\mathbb{R}^n$, $t\in [0,T]$. Note that, due to \eqref{aprioribounds}, $w(x,T)=\frac{4C_o}{\lambda}(1+\lvert x\rvert^2)^{\frac{1}{2}}+\sqrt{2}C_u(1+\lvert x\rvert^2)^{\frac{1}{2}}\geq C_u(1+\lvert x\rvert)\geq u_\lambda(x,T)$. Moreover, using (F1) and (H2),
		$$\begin{array}{l}
			\vspace{0.2cm}\displaystyle{-\partial_t w-\Delta w+\frac{1}{\lambda}H(\lambda Dw,x)+\lambda w-F(x,m_\lambda(t))=-\frac{\sqrt{2}}{2}\lambda C_u e^{-\frac{\lambda}{2}(T-t)}(1+\lvert x\rvert^2)^{\frac{1}{2}}}\\
			\vspace{0.2cm}\displaystyle{\hspace{0.6cm}-\left(\frac{4C_o}{\lambda}+\sqrt{2}C_u e^{-\frac{\lambda}{2}(T-t)}\right)\frac{n(1+\lvert x\rvert^2)-\lvert x\rvert^2}{(1+\lvert x\rvert^2)^{\frac{3}{2}}}+\frac{1}{\lambda}H\left((4C_o+\sqrt{2}\lambda C_u e^{-\frac{\lambda}{2}(T-t)})\frac{x}{(1+\lvert x\rvert^2)^{\frac{1}{2}}},x\right)}\\
			\vspace{0.2cm}\displaystyle{\hspace{0.6cm}+4C_o(1+\lvert x\rvert^2)^{\frac{1}{2}}+\sqrt{2}\lambda C_u e^{-\frac{\lambda}{2}(T-t)}(1+\lvert x\rvert^2)^{\frac{1}{2}}-F(x,m_\lambda(t))}\\
			\vspace{0.2cm}\displaystyle{\hspace{0.3cm}\geq -\left(\frac{4C_o}{\lambda}+\sqrt{2}C_u e^{-\frac{\lambda}{2}(T-t)}\right)\frac{n(1+\lvert x\rvert^2)-\lvert x\rvert^2}{(1+\lvert x\rvert^2)^{\frac{3}{2}}}-\frac{C_o}{\lambda}+4C_o(1+\lvert x\rvert^2)^{\frac{1}{2}}}\\
			\vspace{0.2cm}\displaystyle{\hspace{0.6cm}+\frac{\sqrt{2}}{2}\lambda C_u e^{-\frac{\lambda}{2}(T-t)}(1+\lvert x\rvert^2)^{\frac{1}{2}}-C_o(1+\lvert x\rvert)}\\
			\displaystyle{\hspace{0.3cm}\geq -\frac{C_o(4n+1)}{\lambda}+\sqrt{2}C_u e^{-\frac{\lambda}{2}(T-t)}\left(\frac{\lambda}{2}-n\right)+C_o(1+\lvert x\rvert).}
		\end{array}$$
		Then, for $\lambda\geq 4n+1$, 
		$$-\partial_t w-\Delta w+\frac{1}{\lambda}H(Dw,x)+\lambda w-F(x,m_\lambda(t))\geq C_o(1+\lvert x\rvert)-C_o\geq 0,$$
		namely $w$ is a supersolution to the first equation in \eqref{MFGgen}. By the maximum principle, $$u_\lambda(x,t)\leq \frac{4C_o}{\lambda}(1+\lvert x\rvert^2)^{\frac{1}{2}}+\sqrt{2}C_{u}e^{-\frac{\lambda}{2}(T-t)}(1+\lvert x\rvert^2)^{\frac{1}{2}},$$
		for all $x\in\mathbb{R}^n$, $t\in [0,T]$ and $T>0$. Since for given $(x,t)\in\mathbb{R}^n\times [0,\infty)$ the above estimate holds for every $T>t$, letting $T\to\infty$ we obtain 
		$$u_\lambda(x,t)\leq \frac{4C_o}{\lambda}(1+\lvert x\rvert^2)^{\frac{1}{2}},$$
		for all $x\in\mathbb{R}^n$ and $t\geq0$.
		
		The bound from below in (i) is instead obtained as follows. Let us denote with $\tilde{C}_u:=\underset{(x,t) \in \R^n \times [0,\infty)}{\inf} u_\lambda(x,t)$. Fix $T>0$. We set $v(x,t):=-\frac{L+C_o}{\lambda}+\tilde{C}_ue^{-\lambda(T-t)}$, $x\in\mathbb{R}^n$ and $t\in [0,T]$. From (F2), $\lvert F(x,m)\rvert\leq L$, for all $x\in\mathbb{R}^n$ and $m\in \mathcal{P}_1(\mathbb{R}^n)$. Using this, it is easy to see that $v$ is a subsolution to the first equation in \eqref{MFGgen} and, moreover, $v(x,T)\leq\tilde{C}_u\leq u_\lambda(x,T)$. By the maximum principle, $-\frac{L}{\lambda}+\tilde{C}_ue^{-\lambda(T-t)}\leq u_\lambda(x,t)$, for all $x\in\mathbb{R}^n$, $t\in [0,T]$ and $T>0$. From the arbitrariness of $T$, we conclude that $\tilde{C}_u:=-\frac{L}{\lambda}$.
		
		\vspace{0.3cm}
		
		Now, we prove (ii). Fix $T>0$ and a direction $\xi$ with $\lvert\xi\rvert=1$. We set $\bar{C}_u:=\lVert Du_\lambda\rVert_\infty$. Let us define $v(x,t):=Du_\lambda (x,t)\cdot\xi$, $x\in \mathbb{R}^n$, $t\in [0,T]$. Differentiating with respect to $x$ the equation for $u$ in \eqref{MFGgen} and multiplying it by $\xi$, we find that $v$ solves
		$$-\partial_t v-\Delta v+\frac{1}{\lambda}D_xH(\lambda Du_\lambda,x)\cdot \xi+D_pH(\lambda Du_\lambda,x)\cdot Dv+\lambda v=D_xF(x,m_\lambda(t))\cdot \xi.$$
		Using (F2), (H2) and \eqref{aprioribounds}, 
		$$-\partial_t v-\Delta v+D_pH(\lambda Du_\lambda,x)\cdot Dv+\lambda v\leq L+\frac{C_o}{\lambda}+C_o \bar{C}_u.$$
		Therefore, $v$ is a subsolution to 
		$$-\partial_t w-\Delta w+D_pH(\lambda Du_\lambda,x)\cdot Dw+\lambda w= L+\frac{C_o}{\lambda}+C_o \bar{C}_u.$$
		On the other hand, $w(x,t):=\frac{1}{\lambda}\left(L+\frac{C_o}{\lambda}+C_o \bar C_u\right)+\bar{C}_ue^{-\lambda(T-t)}$, $x\in \mathbb{R}^n$, $t\in [0,T]$, is a supersolution to the above equation and fulfills $v(x,T)\leq \bar{C}_u\leq w(x,T)$. So, by the maximum principle, $$Du_\lambda(x,t)\cdot \xi\leq \frac{1}{\lambda}\left(L+\frac{C_o}{\lambda}+C_o \bar{C}_u\right)+C_ue^{-\lambda(T-t)},$$
		for all $x\in \mathbb{R}^n$, $t\in [0,T]$, $T>0$ and $\xi\in \mathbb{R}^n$ with $\lvert\xi\rvert=1$. Since we can assume $\lvert Du_\lambda(x,t)\rvert>0$, taking $\xi=\frac{Du_\lambda(x,t)}{\lvert Du_\lambda(x,t)\rvert}$ in the above inequality we find
		$\lvert Du_\lambda(x,t)\rvert\leq \frac{1}{\lambda}\left(L+\frac{C_o}{\lambda}+C_o \bar{C}_u\right)+C_ue^{-\lambda(T-t)},$
		for all $x\in \mathbb{R}^n$, $t\in [0,T]$ and $T>0$. From the arbitrariness of $T$, we conclude that
		$$\lvert Du_\lambda(x,t)\rvert\leq \frac{1}{\lambda}\left(L+\frac{C_o}{\lambda}+C_o \bar{C}_u\right),$$
		for all $x\in \mathbb{R}^n$ and $t\geq 0$. Taking the supremum for $(x,t)\in\mathbb{R}^n\times [0,\infty)$, 
		$(\lambda -C_0)\bar{C}_u\leq L+\frac{C_0}{\lambda}.$	Therefore, for $\lambda\geq 2C_o$, $\bar{C}_u\leq \frac{L}{C_o}+\frac{1}{\lambda}$. At this point, we can repeat the above argument with $\frac{L}{C_o}+\frac{1}{\lambda}$ in place of $\bar{C}_u$ to get 
		$\lVert Du_\lambda\rVert_\infty\leq \frac{2(L+1)}{\lambda}.$
		
		\vspace{0.3cm}
		
		Finally, we show (iii). Fix $T>0$. Let\footnote{the shorthand notation $\partial_{x_i} u_\lambda = u_i$, \ldots will be adopted for simplicity.} $z(x,t) = \frac12 |D^2 u_\lambda(x,t)|^2 = \frac12\sum_{i,j} (\partial^2_{ij} u_\lambda(x,t))^2$, which solves (by differentiating the equation for $u_\lambda$ with respect to\footnote{here, we may miss the regularity to differentiate twice and still have an equation satisfied in the classical sense. To make this step completely rigorous, one may adopt a standard approximation argument via smooth data. This technical detail can be addressed similarly whenever it occurs throughout the paper.} $\partial^2_{ij}$ and multiplying it by $\partial^2_{ij} u_\lambda(x,t)$)
		\[
\begin{split}
	-\partial_t z-\Delta z + \displaystyle \sum_{i,j,k} (\partial^3_{ijk} u_\lambda)^2+D_pH \cdot Dz + 2\lambda z=  -\lambda \displaystyle  \sum_{i, j} u_{ij} (D^2_{p}H Du_i) \cdot Du_j \\- 2  \displaystyle  \sum_{i, j} u_{ij} D_pH_{x_j} \cdot Du_i - \frac1\lambda \sum_{i, j} u_{ij} H_{x_i x_j} + \sum_{i, j} u_{ij} F_{ij}.
\end{split}
		\]
		The control from above of $z$ will be obtained by the maximum principle, so we proceed by estimating the right-hand side of the previous equation from above. First, by Cauchy-Schwarz inequality and the assumptions on $H$,
		\[
		- \lambda \sum_{ij} u_{ij} (D^2_{p}H Du_i) \cdot Du_j \le \lambda |D^2 u_\lambda| \Big(\sum_i |D^2_{p}H Du_i|^2\Big)^{1/2} \Big(\sum_j |Du_j|^2\Big)^{1/2} \le \lambda C_o|D^2 u_\lambda|^3.
		\]
		Secondly, by (ii) and an estimate like  \eqref{H3} with $\lambda\lVert Du_\lambda\rVert_\infty\leq 2(L+1)$,
		\[
		- 2 \sum_{ij} u_{ij} D_pH_{x_j} \cdot Du_i \le 2 |D^2 u_\lambda| \Big(\sum_j |D_pH_{x_j}|^2\Big)^{1/2} \Big(\sum_i |Du_i|^2\Big)^{1/2} \le 2C_o(2L+3) |D^2 u_\lambda|^2,
		\]
		and finally, by the assumptions on $H,F$, (and since $\lambda \ge \lambda_0 \ge 1$)
		\[
		- \frac1\lambda \sum_{ij} u_{ij} H_{x_i x_j} + \sum_{ij} u_{ij} F_{ij} \le \left(|D^2_xF| + \frac{|D^2_xH|}{\lambda}\right)|D^2u_\lambda| \le (L + C_o)|D^2 u_\lambda|.
		\]
		Therefore, 
		\[
		-\partial_t z-\Delta z +D_p H\cdot Dz + 2\lambda z \le 4\lambda C_o z^{3/2} + 4C_o(2L+3) z + 2(L+C_o) z^{1/2}.
		\]
		Let us now look $x$-independent supersolutions to the above equation.
		Let $y_1 < y_2$ be steady states, i.e. those values such that
		\[
		4\lambda C_o w^{3/2} + [4C_o(2L+3) - 2\lambda] w + 2(L+C_o) w^{1/2}=4\lambda C_o  w^{\frac{1}{2}}(w^{\frac{1}{2}}-y_1)(w^{\frac{1}{2}}-y_2).
		\]
		 It is easy to see that $y_1\leq \frac{C_1}{\lambda}$, where $C_1>0$ depends on $C_o$, $L$, and that, for large values of $\lambda$ (that is, for $\lambda$ bigger than some constant depending on $C_o$, $L$), we have also $y_2-y_1>1$. Let $\alpha>0$ be such that $\alpha<y_2-y_1-1$. We define $w(x,t):=(y_1+e^{-2\lambda C_o\alpha(T-t)})^2$, $x\in \mathbb{R}^n$ and $t\in [0,T]$. Then,
		$$\begin{array}{l}
			\vspace{0.2cm}\displaystyle{-\partial_t w+ 2\lambda w-4\lambda C_o w^{3/2}- 4C_o(2L+3) w-2 (L+C_o) w^{1/2}=-\partial_t w-4\lambda C_o  w^{\frac{1}{2}}(w^{\frac{1}{2}}-y_1)(w^{\frac{1}{2}}-y_2)}\\
			\vspace{0.2cm}\displaystyle{\hspace{1cm}=-4\lambda C_0e^{-2\lambda C_0\alpha(T-t)}(y_1+e^{-2\lambda C_0\alpha(T-t)})(\alpha +y_1-y_2+e^{-2\lambda C_0\alpha(T-t)})}\\
			\displaystyle{\hspace{1cm}\geq-4\lambda C_0 e^{-2\lambda\alpha(T-t)}(y_1+e^{-2\lambda\alpha(T-t)})(\alpha +y_1-y_2+1)\geq 0.}
		\end{array}$$
		Thus, $w$ is a supersolution. Since $w(x,T)\geq 1\geq z(x,T)$, by the maximum principle $$z(x,t)\leq (y_1+e^{-2\lambda C_o\alpha(T-t)})^2,$$
		for all $x\in \mathbb{R}^n$, $t\in [0,T]$, $T>0$. By definition of $z$, we conclude that $\lVert D^2u_\lambda\rVert_\infty^2\leq 2 y_1^2\leq \frac{2C_1^2}{\lambda^2}$.
		
		Finally, setting $C:=\max\{4C_0,L+C_o,2(L+1),C_1\}$, estimates (i)-(iii) hold true.
	\end{proof}
Now, we give an existence result. In the sequel, $\lambda_0$ will be the positive constant from Proposition \ref{estimatesapriori}.

\begin{prop}\label{basicest} Under the standing assumptions on $F$, $H$ and $m_0$, for any $\lambda \ge \lambda_0$ \eqref{MFGgen} has a classical solution $(u_\lambda,m_\lambda)$ that satisfies \eqref{aprioribounds}. 
\end{prop}
\begin{proof} Solutions can be constructed by taking limits of subsequences of solutions to
\begin{equation}\label{MFGgenT}
\begin{cases}
			-\partial_t u^T_\lambda-\Delta u^T_\lambda+\frac{1}{\lambda}H(\lambda Du^T_\lambda,x)+\lambda u_\lambda^T=F(x,m^T_\lambda(t)),\quad &x\in \mathbb{R}^n,\,t\in (0,T),\\
			\partial_t m^T_\lambda-\Delta m^T_\lambda-\text{div}(m^T_\lambda D_pH(\lambda Du^T_\lambda,x))=0,\quad &x\in \mathbb{R}^n,\,t\in (0,T),\\
			m^T_\lambda(x,0)=m_0(x), \quad u_\lambda^T(x,T)=0, \quad &x\in \mathbb{R}^n, 
		\end{cases}
\end{equation}
as $T \to \infty$ (the existence of a classical solution $(u_\lambda^T, m_\lambda^T)$ for fixed $\lambda, T$ is standard, see for example \cite{CP}). The desired compactness of the sequence $(u_\lambda^T)$ comes from uniform in $T$ estimates of the form \textit{(i)}, \textit{(ii)}, \textit{(iii)} (and additional equicontinuity coming from classical parabolic regularity). Compactness of $(m_\lambda^T)$ comes from estimates of type \eqref{firstcase} below.

One can indeed show that such solutions $(u^T_\lambda,m^T_\lambda)$ satisfy estimates like (i)--(iii) with constants independent of $T$,
 arguing as in the proof of Proposition \ref{estimatesapriori}. In this case, since the final datum $u^T_\lambda(T)=0$, 
it just suffices use the time-independent component of the supersolutions that appear in the proof of Proposition \ref{estimatesapriori}. Hence, \eqref{aprioribounds} is fulfilled. 
\end{proof}
Here below, we tacitly assume that solutions to \eqref{MFGgen} are those that satisfy \eqref{aprioribounds}. We cannot exclude the possibility that other solutions, not satisfying these conditions, exist. Those would fall outside the uniqueness and converge results that are obtained below.
\begin{prop}\label{propgen}
	Let $b_1,b_2:[0,\infty)\times \mathbb{R}^n\times \mathcal{P}_1(\mathbb{R}^n)$ be two continuous and bounded vector fields such that $b_i(t,\cdot,\cdot)$ is Lipschitz continuous in $\mathbb{R}^n\times \mathcal{P}_1(\mathbb{R}^n)$, uniformly with respect to $t$, for any $i=1,2$. Let $m_1$, $m_2$ be solutions to 
	\[
	\partial_t m_i-\Delta m_i-{\rm div}(m_iD_pH(b_i(t,x,m_i),x))=0,\qquad x\in \mathbb{R}^n,\,t\in (r,s).\\
	\]	
	Then, 
	\begin{equation}\label{casogenerale}
		\begin{split}
			&\hspace{-0.5cm}d_1(m_1(s), m_2(s))\\\le& \left(d_1(m_1(r), m_2(r)) +  C_o \int_r^s  \| b_1(t,\cdot,\cdot) - b_2(t,\cdot,\cdot) \|_\infty \dd t\right)e^{ C_o(1+3\tilde{L})(s-r)},
			\end{split}
	\end{equation}
	for all $s,r\geq 0$ with  $s\geq r$, where $\tilde{L}:=\max\{\lVert b_1\rVert_\infty,\text{Lip}(b_1),\lVert b_2\rVert_\infty,\text{Lip}(b_2)\}$ (here, $\text{Lip}(b_i)$ denotes the Lipschitz constant of $b_i(t,\cdot,\cdot)$).
\end{prop}

\begin{proof} Let $\xi_1, \xi_2$ be a pair of random variables in $\R^n$ such that $\mathcal L(\xi_1) = m_1(r)$, $\mathcal L(\xi_2) = m_2(r)$ and
	\[
	d_1(m_1(r), m_2(r)) = \E |\xi_1 - \xi_2|.
	\]
	Moreover, let $X^1$ and $X^2$ be solutions to the SDEs
	\[
	\begin{cases}
		\dd X^i_t = -D_p H(b_i(t,X_t^i,\mathcal{L}(X_t^i)),X^i_t) \dd t + \sqrt2 \dd B_t, \\
		X^i_r = \xi_i \qquad i=1,2.
	\end{cases}
	\]
	Note that $\mathcal L(X^i_t) = m_i(t)$, for all $t \in [r,s]$ and $i=1,2$. By integrating the difference between $X^1_t$ and $X^2_t$ on $(r,s)$ and taking expectations,
	\[
	\E |X^1_s - X^2_s| \le \E |X^1_r - X^2_r| + \int_r^s \E | D_p H(b_1(t,X_t^1,\mathcal{L}(X_t^1)),X^1_t) - D_p H(b_2(t,X_t^2,\mathcal{L}(X_t^2)),X^2_t)  | \dd t.
	\]
	Now, from the assumption (H2), $D_pH(\cdot,x)$ is $C_o$-Lipschitz continuous uniformly in $x$ and, for $\lvert p\rvert\leq K$, $D_pH(p,\cdot)$ is $C_o(1+K)$-Lipschitz continuous uniformly in $p$. Hence, 
	\begin{multline*}
		| D_p H(b_1(t,X_t^1,\mathcal{L}(X_t^1)),X^1_t) - D_p H(b_2(t,X_t^2,\mathcal{L}(X_t^2)),X^2_t)  |  \\
		\le| D_p H(b_1(t,X_t^1,\mathcal{L}(X_t^1)),X^1_t) - D_p H(b_1(t,X_t^2,\mathcal{L}(X_t^2)),X^1_t)  |\\ +|  D_p H(b_1(t,X_t^2,\mathcal{L}(X_t^2)),X^1_t)- D_p H(b_1(t,X_t^2,\mathcal{L}(X_t^2)),X^2_t)  | \\ + | D_p H(b_1(t,X_t^2,\mathcal{L}(X_t^2)),X^2_t) - D_p H(b_2(t,X_t^2,\mathcal{L}(X_t^2)),X^2_t)  |   \\\le
		C_o\lvert b_1(t,X_t^1,\mathcal{L}(X_t^1))-b_1(t,X_t^2,\mathcal{L}(X_t^2))\rvert+C_o(1+\tilde{L})\lvert X^1_t-X^2_t\rvert \\+  C_o\lvert b_1(t,X_t^2,\mathcal{L}(X_t^2)) -b_2(t,X_t^2,\mathcal{L}(X_t^2))\rvert  \\
		\le C_o(1+2\tilde{L})\lvert X^1_t-X^2_t\rvert +C_o\tilde{L}d_1(m_1(t),m_2(t))\\+  C_o \| b_1(t,\cdot,\cdot) - b_2(t,\cdot,\cdot) \|_\infty.
	\end{multline*}
	Therefore, the above estimate together with $d_1(m_2(t),m_2(t))\leq \E\lvert X^1_t-X^2_t\rvert$ gives
	\[
	\E |X^1_s - X^2_s| \le \E |X^1_r - X^2_r| + C_o(1+3\tilde{L})\int_r^s \E\lvert X^1_t-X^2_t\rvert \dd t +  C_o \int_r^s  \| b_1(t,\cdot,\cdot) - b_2(t,\cdot,\cdot) \|_\infty \dd t.
	\]
	By Gr\"onwall's inequality we obtain
	\begin{equation}\label{yyz}
		\E |X^1_s - X^2_s| \le \left(\E |X^1_r - X^2_r|  +  C_o\int_r^s  \| b_1(t,\cdot,\cdot) - b_2(t,\cdot,\cdot) \|_\infty \dd t\right)e^{ C_o(1+3\tilde{L})(s-r)}.
	\end{equation}
	Since $d_1(m_1(s), m_2(s)) \le \E |X^1_s - X^2_s|$ and $\E |X^1_r - X^2_r| =d_1(m_1(r), m_2(r))$, we conclude.
\end{proof}
From Proposition \ref{propgen}, we deduce the following results.
\begin{cor}\label{stimad1}
		Let $\lambda\ge \lambda_0$. Let $m_1,m_2$ be solutions to 
		\[
			\partial_t m_i-\Delta m_i-{\rm div}(m_iD_pH(\lambda Du_i,x))=0,\qquad x\in \mathbb{R}^n,\,t\in (r,s).\\
		\]
	 Then, 
		 \begin{equation}\label{firstcase}
d_1(m_i(s), m_i(r))\le C_o(1+C)(s-r) + \sqrt{2(s-r)},\quad i=1,2,
		\end{equation}
		 and
		\begin{equation}\label{secondcase}
d_1(m_1(s), m_2(s))\le \left(d_1(m_1(r), m_2(r)) +  C_o \int_r^s  \| \lambda D u_1(\cdot,t) - \lambda D u_2(\cdot,t) \|_\infty \dd t\right)e^{ C_o(1+3C)(s-r)},
		\end{equation}
		for all $s,r\geq 0$ with  $s\geq r$ (here, $C$ is as in Proposition \ref{basicest}).

	\end{cor}

\begin{proof} The proof of \eqref{firstcase} is standard, see for instance \cite[Lemma 6]{CP}. Estimate \eqref{secondcase} comes from a direct application of Proposition \ref{propgen}.
\end{proof}
	\begin{cor}\label{uniqm}
		Let $m_1$ , $m_2$ be two solutions to 
		$$\partial_t m_i-\Delta m_i-\text{div}(m_iD_pH(D_x F(x,m_i(t)),x))=0,\quad x\in \mathbb{R}^n,\,t\in (r,s).$$
		Then,
		\begin{equation}\label{limiteq}
			d_1(m_1(s), m_2(s))\le d_1(m_1(r), m_2(r)) e^{ C_o(1+3L)(s-r)},
		\end{equation}
		for all $s,r\geq 0$ with  $s\geq r$.
\end{cor}
Note that, using \eqref{limiteq} with $r=0$, 
$$d_1(m_1(s), m_2(s))\le d_1(m_1(0), m_2(0))  C_o e^{ C_o (1+3L)s},$$
for all $s\geq 0$. In particular, given two solutions $m_1$, $m_2$ to \eqref{MFGlim}, since $m_1$ and $m_2$ share the same initial condition we get $d_1(m_1(s), m_2(s))=0$, i.e. $m_1=m_2$. Hence, under our assumptions \eqref{MFGlim} has a unique solution.	

\begin{cor}\label{estmlambdam}
Let $\lambda\ge \lambda_0$. Let $(u_\lambda,m_\lambda)$ be a solution to \eqref{MFGgen} and let $m$ be the solution to \eqref{MFGlim}. Then,
\begin{equation}\label{mlambdam}
	d_1(m_\lambda(s), m(s))\le \left(d_1(m_\lambda(r), m(r)) +  C_o \int_r^s  \| \lambda D u_\lambda(\cdot,t) - D_xF(\cdot,m(t)) \|_\infty \dd t\right)e^{ C_o (1+3\hat{L})(s-r)},
\end{equation}
for all $s,r\geq 0$ with  $s\geq r$, where $\hat{L}:=\max\{C_0,L\}$.

\end{cor}
The following is the usual representation formula for $u_\lambda$ by Duhamel's principle. It will be used only to get uniqueness of $u_\lambda$ out of uniqueness of $m_\lambda$ and $Du_\lambda$.

\begin{prop}[Representation formula for $u_\lambda$]
		Let $\lambda\geq \lambda_0$. Let $(u_\lambda,m_\lambda)$ be a solution to \eqref{MFGgen}. Then, 
		\begin{equation}\label{urepformulagen}
		{u_\lambda(x,t)}=\int_{t}^{\infty}\int_{\mathbb{R}^n} \Phi(x-y,s-t) e^{-\lambda (s-t)}\left(F(y,m(r))-\frac{1}{\lambda}H(\lambda Du_\lambda(y,s),y)\right)\dd y \dd s,
		\end{equation}
		for all $x\in \mathbb{R}^n$ and $t\geq 0$, where $\Phi(z,r)=\frac{1}{(4\pi r)^{\frac{n}{2}}}e^{-\frac{\lvert z\rvert^2}{4r}}$, for $z\in \mathbb{R}$ and $r>0$.
	\end{prop}
	
	\begin{proof}
		Let $T>0$. Let us first define
		$$v(x,t):=e^{\lambda t}{u_{\lambda}}(x,T-t),\quad \forall x\in \mathbb{R}^n,\forall t\in [0,T].$$
		Then, $v(\cdot,0)=u_\lambda(\cdot,T)$ and 
		$$\partial_tv(x,t)=\lambda e^{\lambda t}u_\lambda(x,T-t)-e^{\lambda t}\partial_t u_\lambda(x,T-t)=\lambda v-e^{\lambda t}\partial_t u_\lambda(x,T-t),\quad (x,t)\in \mathbb{R}^n\times(0,T)$$
		$$\Delta v(x,t)=e^{\lambda t}\Delta u_\lambda(x,T-t),\quad (x,t)\in \mathbb{R}^n\times(0,T).$$
		Now, since $u$ solves the first equation in \eqref{MFGgen}, we obtain that $v$ satisfies 
		\begin{equation}\label{HEgen}
			\begin{cases}
				\partial_tv-\Delta v= f:= e^{\lambda t}\left(F(x,m_\lambda(T-t))-\frac{1}{\lambda}H(\lambda Du_\lambda(x,T-t),x)\right) ,\quad &x\in \mathbb{R}^n, t\in (0,T),\\
				v=u_\lambda(\cdot,T),\quad &\text{on }\mathbb{R}^n\times\{0\},
			\end{cases}
		\end{equation}
		Since \eqref{HEgen} is a Cauchy problem for a inhomogeneous heat equation, we have the following representation formula
		$$v(x,t)=\int_{0}^{t}\int_{\mathbb{R}^n}\Phi(x-y,t-s){F}(y,s)\dd y\dd s+\int_{\mathbb{R}^n}\Phi(x-y,t)u_\lambda(y,T)dy,\quad (x,t)\in \mathbb{R}^n\times [0,T],$$
		where $\Phi(\cdot,\cdot)$ is the fundamental solution of the heat equation given by
		$$\Phi(z,r):=\frac{1}{(4\pi r)^{\frac{n}{2}}}e^{-\frac{\lvert z\rvert^2}{4r}},\quad z\in \mathbb{R}^n,r>0.$$
		Indeed, it is standard that Duhamel's formula hold for $C^{2,1}$ solutions that have polynomial growth in the $x$ variable, together with their derivatives $Dv, D^2v, \partial_t v$, and with right-hand side $f$ having polynomial growth. These growth properties are satisfied in view of Proposition \ref{estimatesapriori} (for $u_\lambda$ and hence for $v$) and the growth assumptions on $F$ and $H$.
		
		Back to $u_\lambda$, we get the following equality
		\begin{multline*}
			u_\lambda(x,t)=e^{-\lambda(T-t)}v(x,T-t)\\
			=e^{-\lambda(T-t)}\int_{0}^{T-t}\int_{\mathbb{R}^n}\frac{e^{-\frac{\lvert x-y\rvert^2}{4(T-t-s)}}e^{\lambda s}}{(4\pi (T-t-s))^{\frac{n}{2}}} \left(F(y,m_\lambda(T-s))-\frac{1}{\lambda}H(\lambda Du_\lambda(y,T-s),y)\right)\dd y\dd s\\
			+e^{-\lambda(T-t)}\int_{\mathbb{R}^n}\Phi(x-y,T-t)u_\lambda(y,T)\dd y,
		\end{multline*}
		for all $x\in \mathbb{R}^n$ and $t\in [0,T]$. Using the change of variable $r=T-s$, we conclude that 
		\begin{multline}\label{repformulaT}
			u_\lambda(x,t)=\int_{t}^{T}\int_{\mathbb{R}^n} \Phi(x-y,s-t) e^{-\lambda (s-t)}\left(F(y,m_\lambda(r))-\frac{1}{\lambda}H(\lambda Du_\lambda(y,s),y)\right)\dd y \dd s\\
			+e^{-\lambda(T-t)}\int_{\mathbb{R}^n}\Phi(x-y,T-t)u_\lambda(y,T)\dd y
		\end{multline}
		for all $x\in \mathbb{R}^n$ and $t\in [0,T]$. 
		\\Now, for given $x\in \mathbb{R}^n$ and $t\geq 0$, the representation formula \eqref{repformulaT} holds for any $T>t$. Note that from property (i) in Proposition \ref{basicest}, 
		$$\left |e^{-\lambda(T-t)}\int_{\mathbb{R}^n}\Phi(x-y,T-t)u_\lambda(y,T)\dd y\right|\leq C_o(1+\lvert x\rvert+C_n )e^{-\lambda(T-t)}\underset{T\to \infty}{\to}0.$$
		Therefore, the second term in the right-hand side of \eqref{repformulaT} goes to zero, as $T\to \infty$. Letting $T\to \infty$ in \eqref{urepformulagen}, we get the assertion.
	\end{proof}
		
	The following representation formula for {$\lambda Du_\lambda$} will be crucial. Note that this is not obtained by mere differentiation of the previous one; in fact, it does not make use of the fundamental solution of the Heat Equation.
	
	\begin{prop}[Representation formula for $\lambda Du_\lambda$]
		Let $\lambda\ge\lambda_0$. Let $(u_\lambda,m_\lambda)$ be a solution to \eqref{MFGgen}.
		Then, for any $T>0$,
		\begin{equation}\label{deru3gen}
				\begin{array}{l}
					\vspace{0.2cm}\displaystyle{\lambda \frac{\partial u_\lambda}{\partial x_i} (x,t)  = \int_{t}^{T}  \int_{\R^n} \Big(\lambda e^{-\lambda (s-t)} \frac{\partial F}{\partial x_i}(y,m_\lambda(s)) - e^{-\lambda (s-t)}\frac{\partial H}{\partial x_i} (\lambda Du_\lambda(y,s),y) \Big) \rho^{x, t}(y,s) \dd y \dd s}\\
					\displaystyle{\hspace{3cm}+\lambda e^{-\lambda T}\int_{\mathbb{R}^n}\frac{\partial u_\lambda}{\partial x_i}(y,T)\rho^{x, t}(y,T)\dd y}.
				\end{array}
		\end{equation}
			for all $i=1,\dots,n$, $x\in \mathbb{R}^n$ and $t\in [0,T]$, where 
		\[
		\begin{cases}
			\partial_t\rho^{x, t}-\Delta \rho^{x, t}-{ \rm div}(\rho^{x, t}D_pH(\lambda Du_\lambda,y))=0,\quad &y\in \mathbb{R}^n,\,{s\geq t},\\
			\rho^{x, t}(t)=\delta_{x}.
		\end{cases}
		\]
	\end{prop} 
	
	\begin{proof} Given $i=1,\dots,n$, differentiating the equation with respect to $i$ and multiplying by $\lambda e^{-\lambda t}$ yields that $z(x,t) = \lambda e^{-\lambda t}\partial_{x_i} u_\lambda(x,t)$ satisfies
	\[
	\begin{cases}
			-\partial_t z-\Delta z+D_p H(\lambda Du_\lambda,x) \cdot Dz + e^{-\lambda t}\partial_{x_i} H(\lambda Du_\lambda,x) = \lambda e^{-\lambda t} \partial_{x_i} F(x,m_\lambda(t)),\, & \text{in $\mathbb{R}^n \times (0,T)$},\\
			z(\cdot,T)=\lambda e^{-\lambda T}\frac{\partial}{\partial x_i}u_\lambda(\cdot,T),\quad &  \text{in $\mathbb{R}^n$}.
		\end{cases}
	\]
	For $x,t$ fixed, let $\rho^{x,t}$ be as in the statement. Testing the equation for $z$ by $\rho$ and the equation for $\rho$ by $z$, and integrating by parts gives
	\begin{multline*}
		\lambda e^{-\lambda t}\partial_{x_i} u_\lambda(x,t) = z(x,t)= \int_{t}^{T} \int_{\R^n} \Big(\lambda e^{-\lambda s} \frac{\partial F}{\partial x_i} (y,m_\lambda(s)) - e^{-\lambda s}\frac{\partial H}{\partial x_i} (\lambda Du_\lambda(y,s),y) \Big) \rho^{x, t}(y,s) \dd y \dd s\\
		+\int_{\mathbb{R}^n}\rho^{x, t}(y,T)z(y,T)\dd y.
	\end{multline*}
	that is, due to $z(y,T)=\lambda e^{-\lambda T}\partial_{x_i} u_\lambda(x,T)$, the desired assertion.
	\end{proof}
	
	Before we move to the next section, let us comment on formula \eqref{deru3gen}, that gives, after letting $T \to \infty$,
	\[
		\displaystyle{\lambda \frac{\partial u_\lambda}{\partial x_i} (x,t)  = \int_{t}^{\infty}  \int_{\R^n} \Big(\lambda e^{-\lambda (s-t)} \frac{\partial F}{\partial x_i} (y,m_\lambda(s)) - e^{-\lambda (s-t)}\frac{\partial H}{\partial x_i}} (\lambda Du_\lambda(y,s),y) \Big) \rho^{x, t}(y,s) \dd y \dd s.
	\]
	On one hand, the second integral on the right-hand side vanishes, since the term $\frac{\partial H}{\partial x_i} (\lambda Du_\lambda(y,s),y) $ is uniformly bounded from (H2) and \eqref{H3}, $\rho^{x,t}(\cdot,s)$ is a probability measure on $\mathbb{R}^n$, for all $s\geq t$, and the integral of the kernel $e^{-\lambda (s-t)}$ vanishes as $\lambda \to \infty$. On the other hand, the first term on the right-hand side can be seen as a (time) convolution with the function $\lambda e^{-\lambda t}{\bf 1}_{[0,\infty)}(t)$; for large $\lambda$, this kernel concentrates at time zero. Moreover, $\rho^{x,t}(t)=\delta_{x}$. For these reasons, we expect that
		\[
		\lambda D u_\lambda(x,t) - D_xF(x,m_\lambda(t))\to 0, \qquad \text{as $\lambda \to \infty$}.
		\]
	This limit will be made rigorous, in the uniform sense, in the next section.
	
	\medskip

	The following lemma is a sort of Gr\"onwall type result, where a further ``future'' in time integral is allowed.
	
	\begin{lemma}\label{groplus} Let $f$ be a nonnegative continuous function on $[0,+\infty)$, and $a, b, \delta > 0$. Assume that
	\[
	f(t) \le a \int_0^t f(s) \dd s + b\int_t^{\infty} f(s) e^{-\delta(s-t)}\dd s, \qquad \forall t \geq 0.
	\]
	If $\delta \ge 4a+4b$ and there exists $M>0$ such that $f(t) \le Me^{4a t}$, for all $t\geq 0$, then $f \equiv 0$.
	\end{lemma}
	
	\begin{proof} Let $t\geq 0$. Since $f(t) \le Me^{4a t}$,
	\begin{multline*}
	f(t) \le a M \int_0^t e^{4a s} \dd s + b M \int_t^{\infty} e^{-\delta(s-t)+4a s} \dd s = \frac{aM}{4a}(e^{4a t} - 1) + \frac{bM}{\delta-4a} e^{4a t} \\ \le
	M\left[\frac{1}{4} +  \frac{b}{\delta-4a} \right]e^{4a t}.
	\end{multline*}
	Since $\delta \ge 4a+4b$, it holds that $\frac{b}{\delta-4a} \le \frac14$. Hence,
	\[
	f(t) \le \frac M2 e^{ 4a t}.
	\]
	Iterating $\ell$ times the previous argument,
	\[
	f(t) \le \frac M{2^\ell} e^{4a t}, \quad \forall \ell\in \mathbb{N}.
	\]
	By letting $\ell \to \infty$ we get $f(t)=0$. Since this is true for any $t\geq 0$, we conclude.
	\end{proof}
	
	We finally present a statement of the standard Gr\"onwall's lemma in backward form.
	
	\begin{lemma}\label{gro} Let $f, g$ be continuous functions on $[0,T]$, $T>0$, and $c, \eta > 0$. If
	\[
	f(t) \le g(t) + c \int_t^T f(r) e^{-\eta(r-t)} \dd r, \qquad \forall t \in [0,T],
	\]
	then
	\[
	f(t) \le g(t) + c \int_t^T g(r) e^{-(\eta-c)(r-t)} \dd r, \qquad \forall t \in [0,T].
	\]
	\end{lemma}

\section{Main results}

The first main result we establish is the convergence of the drift $\lambda Du_\lambda$ to $D_x F(\cdot,m(\cdot))$, where $m$ is the solution to the McKean-Vlasov equation \eqref{MFGlim}. The convergence of $\lambda Du_\lambda$ to $D_x F(\cdot,m(\cdot))$ was already observed in \cite{BC}; here, we get a quantitative version of it. We give its proof to prepare the reader to the uniqueness result, though convergence is not really necessary to achieve uniqueness. The proof is based on the following result, that quantifies the convergence of $\lambda Du_\lambda-D_x F(\cdot,m_\lambda(\cdot))$ to $0$.

\begin{prop}\label{convunifgen}
		Let $\lambda\ge\lambda_0$. Let $(u_\lambda,m_\lambda)$ be a solution to \eqref{MFGgen}. Then there exists $\overline C>0$ depending on $C_o,L$ such that
		\begin{equation}\label{stimaunif}
			\left\| \lambda D u_\lambda-D_xF(\cdot,m_\lambda(\cdot))\right\|_\infty \le \frac{\overline{C}}{\sqrt{\lambda}}\big(L+1\big).
		\end{equation}
		In particular, $\lambda Du_\lambda-DF(\cdot,m_\lambda(\cdot))$ converges uniformly to $0$ on $\mathbb{R}^n\times [0,+\infty)$ as $\lambda \to \infty$.
	\end{prop}
	\begin{proof}
		Let $T>0$. Given $i=1,\dots,n$, $x\in \mathbb{R}^n$ and $t\in [0,T]$, due to \eqref{deru3gen} we have
		\begin{equation}\label{lambdaDu-DFgen}
			\begin{array}{l}
				\vspace{0.3cm}\displaystyle{\left\lvert \lambda \frac{\partial}{\partial x_i}u_\lambda(x,t)-\frac{\partial}{\partial x_i}F(x,m_\lambda(t))\right\rvert}\\
				\vspace{0.3cm}\displaystyle{\hspace{0.3cm}\leq \left\lvert\int_{t}^{T}\int_{\R^n}\lambda e^{-\lambda (s-t)} \frac{\partial}{\partial x_i} F(y,m_\lambda(s)) \rho^{x,t}(y,s)\dd y\dd s-\frac{\partial}{\partial x_i}F(x,m_\lambda(t))\right\rvert}\\
				\displaystyle{\hspace{0.8cm}+\left\lvert\int_{t}^{T}\int_{\R^n}e^{-\lambda (s-t)} \frac{\partial}{\partial x_i} H(\lambda Du_\lambda(y,s),y) \rho^{x,t}(y,s)\dd y\dd s\right\rvert+\lambda e^{-\lambda T}\left\lvert\int_{\mathbb{R}^n}\frac{\partial}{\partial x_i}u_\lambda(y,T)\rho^{x, t}(y,T)\dd y\right\rvert.}
				\end{array}
		\end{equation}
		
	      Let us first estimate the second term of the right-hand side. From \eqref{H3} we have $\lvert D_x H(\lambda Du_\lambda(y,s),y)\rvert\leq C_o(1+C)$, for all $y\in \mathbb{R}^n$ and $s\in [t,T]$. Hence,
		\begin{multline*}
			\left\lvert\int_{t}^{T}\int_{\R^n}e^{-\lambda (s-t)} \frac{\partial}{\partial x_i} H(\lambda Du_\lambda(y,s),y) \rho^{x,t}(y,s)\dd y\dd s\right\rvert\\\le \int_{t}^{T}\int_{\R^n}e^{-\lambda (s-t)} \left\lvert\frac{\partial}{\partial x_i} H(\lambda Du_\lambda(y,s),y)\right\rvert \rho^{x,t}(y,s)\dd y\dd s\\\le C_o(1+C)\int_{t}^{T}e^{-\lambda (s-t)} \int_{\R^n} \rho^{x,t}(y,s)\dd y\dd s=C_o(1+C)\int_{t}^{T}e^{-\lambda (s-t)}\dd s\leq \frac{C_o(1+C)}{\lambda },
		\end{multline*}
	where above we used that $\rho^{x,t}(\cdot,s)$ is a probability measure on $\mathbb{R}^n$, for all $s\in [t,T]$.
	\smallskip
	\\\indent On the other hand, using (ii) in Proposition \eqref{basicest} the last term in the right-hand side of \eqref{lambdaDu-DFgen} can be estimated as follows
	$$\lambda e^{-\lambda T}\left\lvert\int_{\mathbb{R}^n}\frac{\partial}{\partial x_i}u_\lambda(y,T)\rho^{x, t}(y,T)\dd y\right\rvert\leq e^{-\lambda T}\lVert\lambda Du_\lambda\rVert_\infty\int_{\mathbb{R}^n} \rho^{x, t}(y,T)\dd y\leq Ce^{-\lambda T}.$$
	
	\smallskip
	
	Now we address the first term of the right-hand side of \eqref{lambdaDu-DFgen}. First of all, using the fact that $\int_{t}^{T}\int_{\R^n}\lambda e^{-\lambda(s-t)}\rho^{x,t}(y,s)\dd y\dd s=1-e^{-\lambda (T-t)}$, we have
		\begin{multline*}
			\left\lvert\int_{t}^{T}\int_{\R^n}\lambda e^{-\lambda (s-t)} \frac{\partial}{\partial x_i} F(y,m_\lambda(s)) \rho^{x,t}(y,s)\dd y\dd s-\frac{\partial}{\partial x_i}F(x,m_\lambda(t))\right\rvert\\\le \int_{t}^{T}\int_{\R^n}\lambda e^{-\lambda (s-t)} \left\lvert\frac{\partial}{\partial x_i} F(y,m_\lambda(s)) -\frac{\partial}{\partial x_i}F(x,m_\lambda(t))\right\rvert\rho^{x,t}(y,s)\dd y\dd s\\+e^{-\lambda(T-t)}\left\lvert\frac{\partial}{\partial x_i}F(x,m_\lambda(t))\right\rvert.
		\end{multline*}
		Then, the Lipschitz continuity of $D_xF$ yields
		\begin{multline*}
			\left\lvert\int_{t}^{T}\int_{\R^n}\lambda e^{-\lambda (s-t)} \frac{\partial}{\partial x_i} F(y,m_\lambda(s)) \rho^{x,t}(y,s)\dd y\dd s-\frac{\partial}{\partial x_i}F(x,m_\lambda(t))\right\rvert\\\le L\int_{t}^{T}\int_{\R^n}\lambda e^{-\lambda (s-t)} \lvert x-y\rvert\rho^{x,t}(y,s)\dd y\dd s+L\int_{t}^{T}\int_{\R^n}\lambda e^{-\lambda (s-t)} d_1(m_\lambda(s),m_\lambda(t))\rho^{x,t}(y,s)\dd y\dd s\\+e^{-\lambda(T-t)}\left\lvert\frac{\partial}{\partial x_i}F(x,m_\lambda(t))\right\rvert:=I+II+e^{-\lambda(T-t)}\left\lvert\frac{\partial}{\partial x_i}F(x,m_\lambda(t))\right\rvert.
		\end{multline*}
	Note that
	$$I= L\int_{t}^{T}\lambda e^{-\lambda (s-t)}\int_{\R^n}\lvert x -y\rvert ( \rho^{x,t}(y,s) -\rho^{x,t}(y,t))\dd y\dd s\leq L\int_{t}^{T}\lambda e^{-\lambda (s-t)}d_1(\rho^{x,t}(s),\rho^{x,t}(t)).$$
	Then, estimate \eqref{firstcase} applied to $\rho^{x,t}$ yields
	$$I\leq L{C_o(1+C)}\int_{t}^{T}\lambda e^{-\lambda (s-t)}(s-t)+\sqrt{2}L\int_{t}^{T}\lambda e^{-\lambda (s-t)}\sqrt{s-t}:=I_1+I_2.$$
	Using the change of variable $r=\lambda (s-t)$,
	$$I_1=\frac{L{C_o(1+C)}}{\lambda}\int_{0}^{\lambda(T-t)}e^{-r} r \dd r\leq \frac{L{C_o(1+C)}}{\lambda}\int_{0}^{\infty}e^{-r} r \dd r=\frac{L{C_o(1+C)}}{\lambda}.$$
	On the other hand, $$I_2\stackrel{r=\sqrt{\lambda(s-t)}}{=}\frac{2\sqrt{2}L}{\sqrt{\lambda}}\int_{0}^{\sqrt{\lambda(T-t)}}e^{-r^2}r^2\dd r\le \frac{2\sqrt{2}L}{\sqrt{\lambda}}\int_{0}^{\infty}e^{-r^2}r^2\dd r=\frac{\sqrt{2\pi}L}{2\sqrt{\lambda}}.$$
	We can handle the term $II$ similarly. Indeed, using \eqref{firstcase} with {$m_\lambda$} gives
	\[
		\displaystyle{II\le L\int_{t}^{T}\lambda e^{-\lambda (s-t)}({C_o(1+C)}(s-t) + \sqrt{2(s-t)})\int_{\R^n}\rho^{x,t}(y,s)\dd y\dd s}
		\leq \frac{L{C_o(1+C)}}{\lambda}+\frac{\sqrt{2\pi}L}{2\sqrt{\lambda}}.
\]
	Finally, we can conclude that
		\begin{equation}\label{stimaconT}
		\left| \lambda \frac{\partial}{\partial x_i}u_\lambda(x,t)-\frac{\partial}{\partial x_i}F(x,m_\lambda(t))\right| \le \frac{{C_o(1+C)(2L+1)}}{\lambda}+Ce^{-\lambda T}+\frac{\sqrt{2\pi }L}{\sqrt{\lambda}}+e^{-\lambda(T-t)}\left\lvert\frac{\partial}{\partial x_i}F(x,m_\lambda(t))\right\rvert,
		\end{equation}
		for all $x\in \mathbb{R}^n$, $t\in [0,T]$ and $T>0$. Since for given $x\in \mathbb{R}^n$ and $t\geq 0$ the above estimate holds for every $T>t$, letting $T\to \infty$ we obtain
		$$\left| \lambda \frac{\partial}{\partial x_i}u_\lambda(x,t)-\frac{\partial}{\partial x_i}F(x,m_\lambda(t))\right| \le \frac{{C_o(1+C)(2L+1)}}{\lambda}+\frac{\sqrt{2\pi }L}{\sqrt{\lambda}},
		$$
		for all $x\in \mathbb{R}^n$ and $t\geq 0$.
		Taking the supremum for $i=1,\dots,N$ and $(x,t)\in \mathbb{R}^n\times [0,\infty)$, we conclude that
		\begin{equation}\label{ultimastima}
		\| \lambda Du_\lambda-D_x F(\cdot,m_\lambda(\cdot))\|_\infty \le \frac{{C_o(1+C)(2L+1)}}{\lambda}+\frac{\sqrt{2\pi }L}{\sqrt{\lambda}}\leq \frac{\bar{C}}{\sqrt{\lambda}}(L+1),
		\end{equation}
		and so \eqref{stimaunif} is proven. Letting $\lambda\to \infty$, since the right-hand side of  \eqref{stimaunif} goes to 0, we can conclude that the left-hand side of \eqref{stimaunif} goes to 0 too. Therefore, $\lambda Du_\lambda-D_xF(\cdot,m_\lambda(\cdot))$ converges uniformly to $0$ on $\mathbb{R}^n\times [0,+\infty)$, as $\lambda \to \infty$.
		
	\end{proof}
	\begin{rem}[Vanishing viscosity] If one replaces $\Delta$ by $\nu \Delta$ in the MFG system, then for any fixed $\nu_0 > 0$, the constant $\bar C$ in the proposition above depends on $\nu_0$, but it is independent of $\nu \in (0,\nu_0]$. This is because $C$ in Proposition \ref{basicest} does not depend on the viscosity, and \eqref{firstcase} becomes
	\[
	d_1(m_1(s), m_1(r))\le {C_o(1+C)}(s-r) + \sqrt{2\nu(s-r)}.
	\]
	Therefore, one can in fact prove uniform convergence of {$\lambda Du_\lambda$} in the $\lambda \to \infty$ and $\nu = \nu_\lambda \to 0$ limit, as it is done in \cite{BC}. Quantitatively, there seem to be a natural choice $\nu = O(1/\lambda)$, in which case \eqref{ultimastima} would give a convergence rate of order $1/\lambda$.
	\end{rem}
	
	{Now, we state the quantitative version of the convergence result in \cite{BC}, together with quantitative convergence of $m_\lambda$ to the solution of the McKean-Vlasov equation \eqref{MFGlim}.
	\begin{prop}[Uniform Convergence]\label{quantitative}
		Let $\lambda\ge\lambda_0$. Let $(u_\lambda,m_\lambda)$ be a solution to \eqref{MFGgen}. Let $m$ be the solution to \eqref{MFGlim}. Then, for all $\tau>0$ there exists a constant $C_\tau > 0$ depending on $C_o, L, \overline C, \hat L$ (where $\overline{C}$ is the constant in Proposition \ref{convunifgen} and $\hat{L}$ is the constant in Corollary \ref{estmlambdam}) such that
		\begin{equation}\label{stimaunif2}
			\sup_{t\in [0,\tau]}\left\| \lambda D u_\lambda-D_xF(\cdot,m(\cdot))\right\|_\infty \le \frac{C_\tau}{\sqrt{\lambda}},
		\end{equation}
		that is, $\lambda Du_\lambda$ converges uniformly to $D_xF(\cdot,m(\cdot))$ on $\mathbb{R}^n\times [0,\tau]$ as $\lambda \to \infty$, for all $\tau>0$.
		Furthermore,
		\begin{equation}\label{stimaunifm}
		\sup_{t\in [0,\tau]}d_1(m_\lambda(t),m(t)) \le \frac{C_\tau}{\sqrt{\lambda}}.
	\end{equation}
	\end{prop}
\begin{proof}
	Fix $\tau>0$. Given $x\in \mathbb{R}^n$ and $t\in [0,\tau]$, we split
	$$\lvert \lambda D u_\lambda(x,t)-D_xF(x,m(t))\rvert\leq \lvert \lambda D u_\lambda(x,t)-D_xF(x,m_\lambda(t))\rvert+\lvert D_xF(x,m_\lambda(t))-D_xF(x,m(t))\rvert.$$
	Using \eqref{stimaunif},
		$$\lvert \lambda D u_\lambda(x,t)-D_xF(x,m(t))\rvert\leq  \frac{\bar{C}}{\sqrt{\lambda}}(L+1)+\lvert D_xF(x,m_\lambda(t))-D_xF(x,m(t))\rvert.$$
		The Lipschitz continuity of $D_xF$ in (F2) together with \eqref{mlambdam} yields
		\begin{multline}\label{stimamlm}\lvert D_xF(x,m_\lambda(t))-D_xF(x,m(t))\rvert\leq \\ Ld_1(m_\lambda(t),m(t))\leq C_oLe^{ C_o (1+3\hat{L})t} \int_0^t  \| \lambda D u_\lambda(\cdot,t) - D_xF(\cdot,m(s)) \|_\infty \dd s,
		\end{multline}
	where above we have used that $m_\lambda(0)=m_0=m(0)$.
	\\Hence, 
	$$\lvert \lambda D u_\lambda(x,t)-D_xF(x,m(t))\rvert\leq \frac{\bar{C}}{\sqrt{\lambda}}(L+1)+C_oLe^{ C_o (1+3\hat{L})t} \int_0^t  \| \lambda D u_\lambda(\cdot,t) - D_xF(\cdot,m(s)) \|_\infty \dd s,$$
	for all $x\in\mathbb{R}^n$ and $t\in [0,\tau]$. Taking the supremum for $x\in\mathbb{R}^n$ and applying Gronwall's inequality, we finally obtain
	$$\| \lambda D u_\lambda(\cdot,t) - D_xF(\cdot,m(t)) \|_\infty\leq C_oLe^{C_oLte^{C_o(1+3\hat{L})}}\frac{\bar{C}}{\sqrt{\lambda}}(L+1),$$
	for all $t\in [0,\tau]$, which gives \eqref{stimaunif2}.
	
	Finally, \eqref{stimaunifm} follows by plugging \eqref{stimaunif2} back into \eqref{stimamlm}.
	\end{proof} }
%
	
	The next proposition is a crucial step for the uniqueness result. From now on we will drop for brevity the subscript $\lambda$ when referring to solutions to \eqref{MFGgen}.
		
	\begin{prop}\label{dudist}
		Let $(u_1,m_1)$, $(u_2,m_2)$ be two solutions to \eqref{MFGgen}. Then, there exist constants $K, \eta > 0$ depending on $L, C_o$ such that for any $\lambda>\max\{\lambda_0, \eta\}$, $T>0$ and for all $t\in [0,T]$,
	\begin{multline}\label{stimasu0,T}
	\|\lambda D u_1(\cdot,t) - \lambda D u_2(\cdot,t) \|_\infty  \le
	K  d_1(m_1(t), m_2(t)) + Ke^{-\lambda T/2}\\
	+K \int_{t}^{T}  d_1(m_1(r), m_2(r)) e^{-(\lambda - \eta) (r-t)}  \dd r.
\end{multline}
In particular, for all $t\geq 0$,
	\begin{equation}\label{Ttoinfty}
		\|\lambda D u_1(\cdot,t) - \lambda D u_2(\cdot,t) \|_\infty  \le
		K  d_1(m_1(t), m_2(t))			+K \int_{t}^{\infty}  d_1(m_1(r), m_2(r)) e^{-(\lambda - \eta) (r-t)}  \dd r.
\end{equation}
	\end{prop}
	
	\begin{proof} We start by assuming that $\lambda > \lambda_0$. Throughout the proof, the control from below on $\lambda$ will be possibly increased. Fix $T>0$. Let $i=1,\dots,n$, $x\in \mathbb{R}^n$ and $t\in [0,T]$. In view of \eqref{deru3gen},	
	\begin{multline}\label{three}
		\left|\lambda \frac{\partial}{\partial x_i} u_1(x,t) - \lambda \frac{\partial}{\partial x_i} u_2(x,t)\right| \\ \le
		\int_{t}^{T}  \lambda e^{-\lambda (s-t)} \int_{\R^n} \Big| \frac{\partial}{\partial x_i} F(y,m_1(s)) - \frac{\partial}{\partial x_i} F(y,m_2(s)) \Big| \rho_1^{x, t}(y,s) \dd y \dd s  \\
		+\int_{t}^{T} e^{-\lambda (s-t)} \int_{\R^n} \Big| \frac{\partial}{\partial x_i}  H(\lambda Du_1(y,s),y) - \frac{\partial}{\partial x_i}  H(\lambda Du_2(y,s),y) \Big| \rho_1^{x, t}(y,s) \dd y \dd s \\
		+\int_{t}^{T}  \lambda e^{-\lambda (s-t)} \left| \int_{\R^n} \Big( \frac{\partial}{\partial x_i} F(y,m_2(s)) - \lambda^{-1} \frac{\partial}{\partial x_i}  H(\lambda Du_2(y,s),y) \Big)(\rho_1^{x, t}(y,s) - \rho_2^{x, t}(y,s)) \dd y \right| \dd s \\
		+\lambda e^{-\lambda T}\int_{\R^n}\Big| \frac{\partial}{\partial x_i}u_1(y,T)\rho_1^{x, t}(y,T)-\rho_2^{x, t}(y,T)\frac{\partial}{\partial x_i}u_2(y,T)\Big| \dd y  \\ =: I + II + III+IV,
	\end{multline}
	where 
		\[
		\begin{cases}
			\partial_t\rho_i^{x, t}-\Delta \rho_i^{x, t}-{ \rm div}(\rho_i^{x, t}D_pH(\lambda Du_i,y))=0,\quad & \text{on $\mathbb{R}^n \times (t,\infty)$},\\
			\rho_i^{x, t}(t)=\delta_{x}, \qquad i = 1,2.
		\end{cases}
		\]
		
	We now estimate the four integrals $I,II,III, IV$ in the right-hand side of \eqref{three}. First, by the Lipschitz regularity of $D_x F$,
	\[
	I \le L \int_{t}^{T}  \lambda e^{-\lambda (s-t)} d_1(m_1(s), m_2(s)) \int_{\R^n}  \rho_1^{x, t}(y,s) \dd y\dd s=L \int_{t}^{T}  \lambda e^{-\lambda (s-t)} d_1(m_1(s), m_2(s))\dd s,
	\]
	where we have used above that $\rho_1^{x, t}(\cdot,s)$ is a probability measure on $\mathbb{R}^n$. Using now \eqref{secondcase} on the time interval $(s,t)$ we get
	\[
	d_1(m_1(s), m_2(s))\le \left(d_1(m_1(t), m_2(t)) + {C_o} \int_t^s  \| \lambda D u_1(\cdot,r) - \lambda D u_2(\cdot,r) \|_\infty \dd r\right)e^{{C_o(1+3C)}(s-t)}.
\]
	Plugging this estimate in the previous inequality we obtain
	\begin{multline*}
	I \le L \lambda \int_{t}^{T}   e^{-(\lambda - {C_o(1+3C)}) (s-t)} d_1(m_1(t), m_2(t)) \dd s  \\+
	L  {C_o}\lambda  \int_{t}^{T} \int_t^s  \| \lambda D u_1(\cdot,r) - \lambda D u_2(\cdot,r) \|_\infty  e^{-(\lambda -  {C_o(1+3C)} ) (s-t)} \dd r \dd s.
	\end{multline*}
	Note that, assuming $\lambda >  {C_o(1+3C)}$,
	\begin{multline}\label{scambioint}
	\lambda  \int_{t}^{T} \int_t^s  \| \lambda D u_1(\cdot,r) - \lambda D u_2(\cdot,r) \|_\infty  e^{-(\lambda - {C_o(1+3C)}) (s-t)} \dd r \dd s  \\=
	\lambda  \int_{t}^{T}\left( \| \lambda D u_1(\cdot,r) - \lambda D u_2(\cdot,r) \|_\infty \int_r^{T}    e^{-(\lambda - {C_o(1+3C)}) (s-t)}  \dd s \right) \dd r  \\\leq
	\frac{\lambda}{\lambda - {C_o(1+3C)}} \int_{t}^{T} \| \lambda D u_1(\cdot,r) - \lambda D u_2(\cdot,r) \|_\infty e^{-(\lambda -  {C_o(1+3C)}) (r-t)}  \dd r.
	\end{multline}
	Therefore, for $\lambda \ge   {4C_o(1+3C)}$,
	\begin{multline*}
	I \le \frac{L \lambda}{\lambda - {C_o(1+3C)}}  d_1(m_1(t), m_2(t)) \\ + \frac{ L  {C_o}\lambda}{\lambda -  {C_o(1+3C)}}  \int_{t}^{T} \| \lambda D u_1(\cdot,r) - \lambda D u_2(\cdot,r) \|_\infty e^{-(\lambda - \tilde{L} (C+1)) (r-t)}  \dd r  \\\le
	2L  d_1(m_1(t), m_2(t)) + 2 L{C_o}  \int_{t}^{T} \| \lambda D u_1(\cdot,r) - \lambda D u_2(\cdot,r) \|_\infty e^{-(\lambda -  {C_o(1+3C)}) (r-t)}  \dd r.
	\end{multline*}
	
	\smallskip
	Regarding the term $II$, we use the Lipschitz regularity of $D_x H$ from the assumption (H2) and from \eqref{H3} to get
	\[
	II \le C_o(1+C) \int_{t}^{T} e^{-\lambda (s-t)} \| \lambda D u_1(\cdot,s) - \lambda D u_2(\cdot,s) \|_\infty \dd s.
	\]
	
	\smallskip
	
	To handle $III$, we first observe that, by the uniform bounds on {$D^2_{xx} F$, $D^2_{xx} H$, $D^2_{xp}H$ and $\lambda D^2 u$, there exists a constant $\tilde C > 0$ depending on $C, C_o$, $L$}, such that (we use here that $\lambda \ge 1$)
	\begin{multline*}
	\left| \int_{\R^n} \Big( \frac{\partial}{\partial x_i} F(y,m_2(s)) - \lambda^{-1} \frac{\partial}{\partial x_i} H(\lambda Du_2(y,s),y) \Big) (\rho_1^{x, t}(y,s) - \rho_2^{x, t}(y,s)) \dd y \right| \\\le
	 \|D \big(\partial_{x_i} F(\cdot,m_2(s)) - \lambda^{-1} \partial_{x_i} H(\lambda Du_2(\cdot,s),\cdot) \big) \|_\infty d_1\big(\rho_1^{x, t}(s), \rho_2^{x, t}(s)\big) \le \tilde C d_1\big(\rho_1^{x, t}(s), \rho_2^{x, t}(s)\big).
	\end{multline*}
	Moreover, using \eqref{secondcase} with $\rho_1^{x, t}$ and $\rho_2^{x, t}$ that share the same initial condition, and so $d_1(\rho_1^{x, t}(t), \rho_2^{x, t}(t))=0$, we find
	\[
d_1(\rho_1^{x, t}(s), \rho_2^{x, t}(s))\le \ {C_o} e^{ {C_o(1+3C)}(s-t)} \int_t^s  \| \lambda D u_1(\cdot,r) - \lambda D u_2(\cdot,r) \|_\infty \dd r.
	\]
	Therefore,
	\[
	III \le  \tilde C {C_o} \lambda \int_t^{T}  \int_t^s e^{-(\lambda - {C_o(1+3C)}) (s-t)}   \| \lambda D u_1(\cdot,r) - \lambda D u_2(\cdot,r) \|_\infty\dd r  \dd s.
	\]
	Arguing as in \eqref{scambioint}, we get that
	\[
	III \le 2 \tilde C  {C_o} \int_{t}^{T} \| \lambda D u_1(\cdot,r) - \lambda D u_2(\cdot,r) \|_\infty e^{-(\lambda - {C_o(1+3C)}) (r-t)}  \dd r.
	\]
As far as the term $IV$ is concerned, thanks to the estimates on $\|\lambda Du_i\|$,
		\[
			IV\le \lambda e^{-\lambda T}\int_{\R^n}\Big| \frac{\partial}{\partial x_i}u_1(y,T) \Big|\rho_1^{x, t}(y,T) \dd y+ \lambda e^{-\lambda T}\int_{\R^n}\Big| \frac{\partial}{\partial x_i}u_2(y,T) \Big|\rho_2^{x, t}(y,T) \dd y \le 2Ce^{-\lambda T}.
		\] 

\smallskip

Finally, putting the estimates on $I, II, III, IV$ together in \eqref{three} we obtain
	\begin{multline*}
	\left|\lambda \frac{\partial }{\partial x_i}u_1(x,t) - \lambda \frac{\partial }{\partial x_i} u_2(x,t)\right| \le
	2L  d_1(m_1(t), m_2(t))+ 2Ce^{-\lambda T} \\+
	{C_o(1+C+2\tilde{C}+2L)} \int_{t}^{T} \| \lambda D u_1(\cdot,r) - \lambda D u_2(\cdot,r) \|_\infty e^{-(\lambda - {C_o(1+3C)}) (r-t)}  \dd r.
	\end{multline*}
	Taking then  the supremum with respect to $i=1,\dots,n$ and $x\in \mathbb{R}^n$ gives
	\begin{multline*}
	 \| \lambda D u_1(\cdot,t) - \lambda D u_2(\cdot,t) \|_\infty  \le
	2L  d_1(m_1(t), m_2(t))  + 2Ce^{-\lambda T}\\+
	{C_o(1+C+2\tilde{C}+2L)}\int_{t}^{T} \| \lambda D u_1(\cdot,r) - \lambda D u_2(\cdot,r) \|_\infty e^{-(\lambda -  {C_o(1+3C)}) (r-t)}  \dd r.
	\end{multline*}
	
         The application of Gr\"onwall's Lemma \ref{gro} gives now the desired assertion \eqref{stimasu0,T}
	\begin{multline*}
	 \| \lambda D u_1(\cdot,t) - \lambda D u_2(\cdot,t) \|_\infty  \le
	2L  d_1(m_1(t), m_2(t)) + (1+ T {C_o(1+C+2\tilde{C}+2L)}) 2Ce^{-\lambda T}\\
	+{C_o(1+C+2\tilde{C}+2L)} \int_{t}^{T}  d_1(m_1(r), m_2(r)) e^{-(\lambda -{2C_o(1+2C+\tilde{C}+L)}) (r-t)}  \dd r,
	\end{multline*}
	where $\eta$ is chosen to be the maximum between $2C_o(1+2C+\tilde{C}+L)$  and $4C_o(1+3C)$ and $K$ is chosen to be the maximum between $C_o(1+C+2\tilde{C}+2L)$, $2L$ and $1+C_0(1+2C+\tilde{C}+L)C'$, where $C' = C'(\lambda_0)$ is a universal constant such that $T e^{-\lambda T} \le C' e^{-\lambda T/2}$. Finally, for all $t\geq 0$, letting $T\to \infty$ in \eqref{stimasu0,T}, we conclude that \eqref{Ttoinfty} is fulfilled.
	\end{proof}
	
	We are now ready to state and prove the main uniqueness result.
	
	\begin{thm}\label{maingen}
	There exists a constant $\eta'$ depending on $L,C_o$ such that, if $\lambda\ge\Lambda = \max\{\lambda_0, \eta'\}$, then \eqref{MFGgen} has a unique solution, in the following sense: given two couples $(u_1,m_1)$, $(u_2,m_2)$ solving \eqref{MFGgen}, and satisfying condition \eqref{aprioribounds} of Proposition \ref{estimatesapriori}, then $u_1 \equiv u_2$ and $m_1 \equiv m_2$.
\end{thm}

	\begin{proof} We start by assuming that $\lambda > \max\{\lambda_0, \eta, {4C_o(1+C)}\}$, where $\lambda_0, \eta$ are as in previous propositions. Let $T>0$, and  $(u_1,m_1)$, $(u_2,m_2)$ be two solutions to \eqref{MFGgen}.  Let $\xi$ be a random variable in $\R^n$ such that $\mathcal L(\xi) = m_0$, and let $X^1$, $X^2$ be solutions to the SDEs
\[
\begin{cases}
\dd X^i_t = -D_p H(\lambda D u_i(X^i_t,t),X^i_t) \dd t + \sqrt2 \dd B_t \\
X^i_0 = \xi \qquad i=1,2.
\end{cases}
\]
Note that $\mathcal L(X^i_t) = m_i(t)$ for all $t \in [0,T]$ and $i=1,2$. Arguing as in the proof of Proposition \ref{propgen} and using that $X^1_0=\xi=X^2_0$, we get, for all $s \in [0,T]$,
\[
\E |X^1_s - X^2_s| \le {C_o(1+3C)}\int_0^s {\E}\lvert X^1_t-X^2_t\rvert \dd t + {C_o} \int_0^s  \| \lambda D u_1(\cdot,t) - \lambda D u_2(\cdot,t) \|_\infty \dd t.
\]
Since $d_1(m_1(t), m_2(t)) \le \E |X^1_t - X^2_t|$, by the estimate \eqref{stimasu0,T} in Proposition \ref{dudist} we have
	\begin{multline}\label{stimaE}
		\E |X^1_s - X^2_s| \le {C_o(1+3C+K)} \int_0^s\E \lvert X^1_t-X^2_t\rvert \dd t + {C_oK} \int_0^s  \int_{t}^{T} \E \lvert X^1_r-X^2_r\rvert e^{-(\lambda - \eta) (r-t)}  \dd r  \dd t\\
		+ KTe^{-\lambda T/2}.
	\end{multline}

	We focus now on the second term of the right-hand side, that can be manipulated as follows
	\begin{multline*}
	\int_0^s  \int_{t}^{T} \E \lvert X^1_r-X^2_r\rvert e^{-(\lambda - \eta) (r-t)}  \dd r  \dd t  \\ =\int_0^s  \E \lvert X^1_r-X^2_r\rvert\int_{0}^{r} e^{-(\lambda - \eta) (r-t)}  \dd t  \dd r + 
	     \int_s^T   \E \lvert X^1_r-X^2_r\rvert \int_{0}^{s} e^{-(\lambda - \eta) (r-t)}  \dd t  \dd r \\
	     \le \frac1{\lambda - \eta} \int_0^s  \E \lvert X^1_r-X^2_r\rvert \dd r + \frac1{\lambda - \eta} \int_s^T  \E \lvert X^1_r-X^2_r\rvert e^{-(\lambda - \eta) (r-s)}  \dd r, \\
	     \le  \int_0^s  \E \lvert X^1_r-X^2_r\rvert \dd r +  \int_s^T  \E \lvert X^1_r-X^2_r\rvert e^{-(\lambda - \eta) (r-s)}  \dd r,
	\end{multline*}
	where in the last inequality we used that $\lambda \ge \eta +1$. Plugging the previous estimate into \eqref{stimaE} we find that 
	\begin{multline*}
			\E |X^1_s - X^2_s| \le {C_o(1+3C+2K)} \int_0^s\E \lvert X^1_t-X^2_t\rvert \dd t + {C_oK} \int_s^{T}  \E \lvert X^1_r-X^2_r\rvert e^{-(\lambda - \eta) (r-s)}  \dd r\\
			+KTe^{-\lambda T/2},
	\end{multline*}
	for all $t\in [0,T]$ and for any $T>0$. Taking the limit $T \to \infty$ we finally obatin
	\[\E |X^1_s - X^2_s| \le {C_o(1+3C+2K)} \int_0^s\E \lvert X^1_t-X^2_t\rvert \dd t + {C_oK}\int_s^{\infty}  \E \lvert X^1_r-X^2_r\rvert e^{-(\lambda - \eta) (r-s)}  \dd r,
	\]
	for all $s\geq 0$.

	We now check that Lemma \ref{groplus} applies to $f(t) = \E |X^1_t - X^2_t|$. Assuming that
	\[
	\lambda - \eta \ge {4C_o(1+3C+2K)+4C_oK} ,
	\]
	we just need to show that $ \E |X^1_t - X^2_t| \le M e^{ {4C_o(1+3C+2K)}t}$, for all $t$. This is a consequence of Proposition \ref{stimad1} (see in particular \eqref{yyz}) and the estimates on $\lambda Du_i$ : for all $t \ge 0$,
	\[
\E |X^1_t - X^2_t| \le C_oe^{ C_o(1+3C)t} \int_0^t  \| \lambda D u_1(\cdot,t) - \lambda D u_2(\cdot,t) \|_\infty \dd t \le 2Ct C_oe^{ C_o(1+3C)t}  \le 2CC_o e^{ 2C_o(1+3C)t}.
\]
	Therefore, $d_1(m_1(t), m_2(t)) \le \E |X^1_t - X^2_t| = 0$ for all $t$, hence $m_1 \equiv m_2$. To get this final step, we need to assume that $\lambda \ge \eta' := \eta + {4C_o(1+3C+2K)+4C_oK}$.
	\smallskip
	
		Now, estimate \eqref{Ttoinfty} in Proposition \ref{dudist} together with $m_1=m_2$ yields  $\|\lambda D u_1(\cdot,t) - \lambda D u_2(\cdot,t) \|_\infty =0$, for all $t\geq 0$. Hence, $\lambda Du_1=\lambda Du_2$. From the representation formula \eqref{urepformulagen}, using that $m_1=m_2$ and $\lambda Du_1=\lambda Du_2$, we obtain $u_1-u_2=0$, namely $u_1=u_2$.

	\end{proof}
	
	\begin{rem}[Uniqueness for finite horizon problems] Uniqueness of solutions is obtained for \eqref{MFGgen}, that is set on the time horizon $[0,\infty)$. From propositions \ref{estimatesapriori} and \ref{basicest}, solutions are naturally those that come from limits of the same system of PDE on $[0,T]$, where the final condition $u_\lambda(T) = 0$ is imposed, see \eqref{MFGgenT}. We observe that the same uniqueness result holds for such a family of problems, and the uniqueness threshold for $\lambda$ is in fact independent of $T$; this can be obtained arguing in the very same way (it just suffices not to take the limit for $T \to \infty$ in the previous proofs and to notice that in estimate \eqref{stimasu0,T} the term $Ke^{-\lambda \frac{T}{2}}$ does not appear due to $u_\lambda(T)=0$). On the other hand, the uniform convergence result stated in Proposition \ref{convunifgen} holds on time intervals $[0, \tau] \subsetneqq [0,T]$ (see estimate \eqref{stimaconT}); this is expected, since $\lambda Du_\lambda$ needs to deviate from $D_xF(\cdot,m(\cdot))$ close to $T$ to reach the final condition. 
	
	One could certainly replace the final condition on $u_\lambda(T)$ with nonzero functions. We cannot exclude the possibility that other choices may lead to a completely different behavior of the system.
	
	\end{rem}


\end{document}